\documentclass[reqno,11pt,20pt]{amsart}
\usepackage[ansinew]{inputenc}
\usepackage[english]{babel}
\usepackage[all]{xy}
\usepackage{amsmath,latexsym,amssymb,verbatim}
\usepackage{ stmaryrd }

\input arrow.tex

\newcommand{\LL}{{\mathbb L}}
\newcommand{\FF}{{\mathtt F}}
\newcommand{\M}{{\mathcal M}}

\newcommand{\CR}{{\mathcal R}}

\newcommand{\ZZ}{{\mathbb Z}}

\newcommand{\RR}{\mathbb R}
\newcommand{\0}{\mathbf 0}
\newcommand{\un}{\mathbf 1}

\newcommand{\mm}{{\mathfrak m}}

\DeclareMathOperator{\limi}{{lim}}
\newcommand{\plim}[1]{\,\underset{#1}{\underset{\leftarrow}{\limi}}\,}

\newcommand{\enumera}{\begin{enumerate}}
\newcommand{\eenumera}{\end{enumerate}}

\newcommand{\A}{{\mathbb A}}

\newcommand{\stella}{{\scriptscriptstyle\bigstar}}
 
\DeclareMathOperator{\Hom}{{Hom}}

\DeclareMathOperator{\Res}{{Res}}
\DeclareMathOperator{\di}{{d}}

\newcommand{\ilim}[1]{\,\underset{#1}{\underset{\to}{\limi}}\,}
\newcommand{\alineas}[1]{\begin{array}{#1}}
\newcommand{\alinea}{\begin{array}{l}}
\newcommand{\ealinea}{\end{array}}
\newcommand{\ealineas}{\end{array}}

\newcommand{\pun}{{\scriptscriptstyle \bullet}}

\newcommand{\HHom}{\mathcal{H}om}
\DeclareMathOperator{\Link}{Link}

\newcommand{\Shv}{\operatorname{Shv}}

\newcommand{\supp}{\operatorname{supp}}

\DeclareMathOperator{\Ext}{{Ext}}

\DeclareMathOperator{\id}{{id}}

\theoremstyle{plain}
\newtheorem{thm}{Theorem}[section]
\newtheorem{lem}[thm]{Lemma}

\newtheorem{prop}[thm]{Proposition}
\newtheorem{defn}[thm]{Definition}
\newtheorem{rem}[thm]{Remark}
\newtheorem{rems}[thm]{Remarks}

\newtheorem{cosa}[thm]{}

\numberwithin{equation}{thm}

\begin{document}

\title{Hochster's type formulae}

\author{Fernando Sancho de Salas and Alejandro Torres Sancho}

\address{ Departamento de
Matem\'aticas and Instituto Universitario de F\'isica Fundamental y Matem\'aticas (IUFFyM)\newline
Universidad de Salamanca\newline  Plaza de la Merced 1-4\\
37008 Salamanca\newline  Spain}
\email{fsancho@usal.es}
\email{atorressancho@usal.es}



\thanks {Work supported by Grant PID2021-128665NB-I00 funded by MCIN/AEI/ 10.13039/501100011033 and, as appropriate, by ``ERDF A way of making Europe''.}
 
\begin{abstract}   We give an elementary proof and generalize some Hochsters's type formulae on local cohomology and Ext's of squarefree modules.
\end{abstract}



\maketitle

\section*{Introduction} 

Let $k$ be a  commutative ring and $R=k[x_1,\dots,x_n]$ the polynomial ring over $k$. Let $\A^n_{\FF_1}$ be the finite poset of subsets of the finite set $\Delta_n=\{1,\dots, n\}$ and let us denote by $\Shv_k(\A^n_{\FF_1})$ the category of sheaves of $k$-modules on $\A^n_{\FF_1}$. In \cite{Ya01},  Yanagawa shows that a sheaf of $k$-modules on $\A^n_{\FF_1}$ has an associated squarefree $R$-module; this defines an exact functor:
\[ \pi^\stella\colon \Shv_k(\A^n_{\FF_1})\longrightarrow  R-\text{mod.} \] (where $R-\text{mod.}$ denotes the category of all $R$-modules) which maps each sheaf $F$ on $\A^n_{\FF_1}$ to its associated squarefree module. This functor has an exact right  adjoint
\[\pi_\stella\colon  R-\text{mod.} \longrightarrow  \Shv_k(\A^n_{\FF_1}).\] 
Let  $D(R)$ be the derived category of complexes of $R$-modules and $D(\A^n_{\FF_1})$ the derived category of complexes of sheaves of $k$-modules on $\A^n_{\FF_1}$. A complex $F\in  D(\A^n_{\FF_1})$ is called {\em pointwise perfect} if its stalk $F_p$ at any point $p\in\A^n_{\FF_1}$ is a perfect complex of $k$-modules, i.e., $F_p$  is quasi-isomorphic to a bounded complex of finitely generated projective $k$-modules. 

Let $k_{\{\un\}}$ be the constant sheaf $k$ on $\A^n_{\FF_1}$ supported at the point $\un\in\A^n_{\FF_1}$ that represents the total subset of $\Delta_n$. For any $F\in D(\A^n_{\FF_1})$, let us denote
\[ D(F)=\RR\HHom_{\A^n_{\FF_1}}^\pun(F, k_{\{\un\}})\] and for any $M\in D(R)$ let us denote
\[ D(M)=\RR\Hom_R^\pun(M,\Omega^n_{R/k})\] where $\Omega^n_{R/k}$ is the module of highest K\"ahler differentials of $R$ over $k$. 
  The main theorem of this paper is the following:\medskip

\noindent{\bf Theorem.}[Thm. \ref{Hochster}] For any $ M\in D (R)$, and any pointwise perfect complex $F\in D (\A^n_{\FF_1})$, one has
\[ \RR\Hom_R^\pun( M, \pi^\stella F)=\RR\Hom_{\A^n_{\FF_1}}^\pun(D(F),\pi_\stella D(M)).\]

We shall see that this theorem yields Hochster's type formulae whenever one is able to decompose $\pi_\stella D(M)$ as a direct sum, and this of course highly depends on the choice of $M$. The first example of this is given in Theorem \ref{Hochster-injective}, obtaining a Hochster  formula when $k$ is a field and $D(M)=I[r]$ for some injective $R$-module $I$ and some integer $r$.

Let us denote $\mm=(x_1,\dots,x_n)$ and $\mm_l=(x_1^l,\dots,x_n^l)$. In Theorem  \ref{cor2} we give a Hochster formula for the local cohomology modules $H^i_\mm(\pi^\stella F)$, which generalizes Yanagawa's results, stated for a sheaf $F$ of finite dimensional vector $k$-spaces, where $k$ is a field  (\cite[Thm. 3.3, Thm. 3.5]{Ya03}, see also \cite[Thm. 1.1]{BBR})  for any ring $k$ and any pointwise perfect complex $F$ of $k$-modules. In this case, the formula is a consequence of Theorem \ref{Hochster} and of the decomposition of $\pi_\stella H^n_\mm(\Omega^n_{R/k})$ given in Proposition \ref{residuo-descomp}, (a). We also give the explicit $R$-module structure, in terms of residue maps, of the right hand term  $\underset{p\in \overline{\supp(F)}}\bigoplus\, H^{i-c_p}_p(U_p,F)\otimes_k H^{c_p}_\mm(R_{C_p})$, in Remarks \ref{residuo}, (2).

In Theorem \ref{Hochster-extens} we give a Hochster formula for $\Ext_R^i(R/\mm_l,\pi^\stella F)$  for any pointwise perfect complex $F$. As a particular case, we obtain  Miyazaki's result (\cite[Thm. 1]{Mi}). In this case, the formula is a consequence of Theorem \ref{Hochster} and of the decomposition of $\pi_\stella \Hom_R(R/\mm_l, H^n_\mm(\Omega^n_{R/k}))$ given in Proposition \ref{residuo-descomp}, (b).

Finally, let us mention that the decompositions of Propostion \ref{residuo-descomp} are easily obtained by reduction to the case $n=1$, thanks to the additivity of $\pi_\stella$ (Proposition \ref{product}).

In the last thirty years, many Hochster's type formula have appeared in the literature (see specially \cite{ABZ}, where many Hochster's type formulae are obtained and compared with previous work on the subject), essentially providing Hochster's type formulae for arbitrary posets. We shall see at the end of the paper how these results essentially differ from ours.

%
%

\section{Preliminaries and notations}

\begin{cosa}[\bf The finite poset $\A^n_{\FF_1}$] {\rm  Let $\Delta_{n}=\{ 1,\dots,n\}$ be the finite set with $n$ elements. 

\begin{defn}{\rm We shall denote $$\A^n_{\FF_1}:=\mathcal{P}(\Delta_{n})$$ the set of subsets of $\Delta_{n}$, which is a finite topological space (a finite poset) with the topology given by the partial order defined by inclusion: $p\leq q\Leftrightarrow p\subseteq q$. Any finite poset $(X,\leq)$ will be considered as a topological space with its Alexandroff topology: a subset $U$ of $X$ is open if $p\in U$ and $q\geq p$ implies $q\in U$.}
\end{defn} We shall denote:

\begin{enumerate}

\item[$\bullet$] $\0\in\A^{n}_{\FF_1}$ the element representing the empty subset of $\Delta_{n}$. It is the unique closed point of $\A^{n}_{\FF_1}$.

\item[$\bullet$] $\un\in\A^{n}_{\FF_1}$ the generic point of $\A^n_{\FF_1}$, i.e., the element representing the total subset of $\Delta_{n}$. It is the unique open point of $\A^{n}_{\FF_1}$.

\item[$\bullet$] For each $p\in\A^n_{\FF_1}$, we denote
\[ \aligned U_p&=\text{ smallest open subset containing } p =\{ q\in\A^n_{\FF_1}: q\geq p\}
\\ C_p &=  \text{ closure of } p =\{ q\in\A^n_{\FF_1}: q\leq p\}
\endaligned\]

\item[$\bullet$] For each $i\in\Delta_{n}$, we shall denote by $U_i$ the open subset of $\A^{n}_{\FF_1}$ whose elements are the subsets of $\Delta_{n}$ containing $i$; in other words $U_i=U_{\{i\}}$. 

\item[$\bullet$] We shall denote   $H_i:=\A^{n}_{\FF_1}- U_i$ the hyperplane of $\A^{n}_{\FF_1}$ whose elements are the subsets of $\Delta_{n}$ contained in $\Delta_{n}-\{ i\}$. 

\item[$\bullet$] For any   $p\in \A^{n}_{\FF_1}$, $\widehat p$ denotes the complement subset: $\widehat p=\Delta_n - p$. 
\end{enumerate}


A finite simplicial complex $K$ (with at most $n$ vertices) is just a closed subset of $\A^n_{\FF_1}$. Recall that for any simplical complex $K\subseteq \A^n_{\FF_1}$ and any $p\in K$, the link of $p$ in $K$ is:
\[\Link_K(p)=\{ q\in K: q\neq \0, p\cap q=\0\text{ and } p\cup q\in K\}\] which is a closed subset of $K^*=K-\{\0\}$. If we denote $U_p^*=U_p-\{p\}$, one has an homeomorphism (equivalently, an order preserving bijection) $$\aligned \Link_K(p)&\overset\sim\longrightarrow   U_p^*\cap K\\ q&\longmapsto p\cup q\\ x-p&\longmapsfrom x\endaligned.$$

We shall denote by $\Shv_k(\A^n_{\FF_1})$ the category of sheaves of $k$-modules on $\A^n_{\FF_1}$ and by $D(\A^n_{\FF_1})$ the derived category of complexes of sheaves of $k$-modules on $\A^n_{\FF_1}$. We shall use the standard notations of the derived category of sheaves on a topological space. 
As it is well known, a sheaf $F$ of $k$-modules on $\A^n_{\FF_1}$ is equivalent to give, for each $p\in\A^n_{\FF_1}$, a $k$-module $F_p$, and for each $p\leq q$ a morphism of $k$-modules $r_{pq}\colon F_p\to F_q$, such that $r_{pp}=\id_{F_p}$ and $r_{ql}\circ r_{pq}=r_{pl}$ for any $p\leq q\leq l$. 

The constant sheaf $k$ on $\A^n_{\FF_1}$ is still denoted by $k$. For any $p\in \A^n_{\FF_1}$ we denote
\[ k_{U_p}\quad,\quad k_{C_p}\quad,\quad k_{\{p\}}\] the constant sheaf $k$ supported, respectively, on $U_p$, $C_p$ and $\{p\}$. We shall use the basic formulae
\[ \aligned &\RR\Hom_{\A^n_{\FF_1}}^\pun(k_{U_p}, F) = F_p\\ &\RR\Hom_{\A^n_{\FF_1}}^\pun(F, k_{C_p} ) =\RR\Hom_k^\pun(F_p,k)\\ &\RR\Hom_{\A^n_{\FF_1}}^\pun(k_{\{p\}}, F)=\RR\Gamma_p(U_p,F),\endaligned\] where $\RR\Hom^\pun$ is the right derived functor of $\Hom^\pun$ (the  complex of homomorphisms)  and  $\Gamma_p(U_p,\quad)$ denotes the functor of sections with support in $p$.

One has a natural isomorphism $$\A^n_{\FF_1}=\A^1_{\FF_1}\times \overset n\cdots\times \A^1_{\FF_1}$$
so an element $p\in \A^n_{\FF_1}$ is identified with a sequence $p=(p_1,\dots,p_n)$ with $p_i\in \A^1_{\FF_1}=\{\0,\un\}$. Explicitely, $p_i=\0$ if $i\notin p$ and $p_i=\un$ if $i\in p$.

For each $i=1,\dots , n$, we have the $i$-th projection 
\[\phi_i\colon \A^n_{\FF_1}\longrightarrow \A^1_{\FF_1}\] defined by $\phi_i^{-1}(\0)=H_i$  (and then $\phi_i^{-1}(\un)=U_i$). \medskip

If $F_1,\dots , F_n$ are sheaves on $\A^1_{\FF_1}$ we have a sheaf on $\A^n_{\FF_1}$:
\[ F_1\boxtimes \overset n\cdots\boxtimes F_n := \phi_1^{-1}F_1\otimes_k \cdots\otimes_k \phi_n^{-1}F_n \] whose stalkwise description is
\[ (F_1\boxtimes \overset n\cdots\boxtimes F_n)_p =(F_1)_{p_1}\otimes_k\cdots \otimes_k (F_n)_{p_n}\] with $p=(p_1,\dots ,p_n)$, via the equality $\A^n_{\FF_1}=\A^1_{\FF_1}\times \cdots\times \A^1_{\FF_1}$.

}\end{cosa}

\begin{cosa}[\bf Pointwise perfect complexes and duality] \label{pointwise-perfect}{\rm A complex $F\in D(\A^n_{\FF_1})$ is called {\em pointwise perfect} if $F_p$ is a perfect complex of $k$-modules for any $p\in \A^n_{\FF_1}$. For example, if $k$ is a field, then $F$ is pointwise perfect if and only if $F_p$ is a complex of $k$-vector spaces with bounded and finite dimensional cohomology. We shall denote by $D_{\text{\rm perf}}(\A^n_{\FF_1})$ the subcategory of pointwise perfect complexes. 

For any ring $k$, $D(k)$ denotes the derived category of complexes of $k$-modules. For any $E\in D(k)$, we   denote $E^\vee=\RR\Hom_k^\pun(E,k)$. One has a natural morphism $E\to E^{\vee\vee}$ which is an isomorphism if $E$ is perfect. For any $E\in D(k)$, we still denote by $E$ the constant complex in $D(\A^n_{\FF_1})$ defined by $E_p=E$ for any $p\in\A^n_{\FF_1}$ and the identity morphism $E_p\to E_q$ for any $p\leq q$.

For each $F\in D(\A^n_{\FF_1})$ we shall denote
\[ D(F):=\RR\HHom_{\A^n_{\FF_1}}^\pun(F,k_{\{\un\}}),\] 
where $\HHom_{\A^n_{\FF_1}}^\pun$ denotes the complex of sheaves of homomorphisms and $\RR\HHom_{\A^n_{\FF_1}}^\pun$ its right derived functor.

We shall use the following property (see \cite{ST1}):

If $F$ is pointwise perfect, so is $D(F)$, and for any $p\in\A^n_{\FF_1}$ and any $E\in D(k)$, one has:
\begin{equation}\label{formulaD(F)} \RR\Hom_{\A^n_{\FF_1}}^\pun(D(F),E\otimes_kk_{C_p})=\RR\Gamma_p(U_p,F) \overset\LL\otimes_k E\, [n-c_p ].\end{equation}

The {\em support} of a complex $F\in D(\A^n_{\FF_1})$ is the set
\[ \supp(F)=\{ p\in \A^n_{\FF_1}: F_p\neq 0\,\, (\text{in } D(k))\}.\] We shall denote by $\overline{\supp(F)}$ the closure of $\supp(F)$. Notice that $F_{\vert U_p}=0$ if $p\notin  \overline{\supp(F)}$.
}\end{cosa}

\begin{cosa}[\bf Yanagawa's functor and its right adjoint] {\rm  Let $R=k[x_1,\dots,x_n]$. If $M$ is a $\ZZ^n$-graded $R$-module (or $k$-module or a sheaf of $k$-modules) and $\alpha=(\alpha_1,\dots,\alpha_n)\in\ZZ^n$, then $M(\alpha)$ denotes the module $M$ shifted by $\alpha$, that is, it is the same  module $M$ but with the gradation
\[ [M(\alpha)]_{(m_1,\dots,m_n)}=M_{(m_1+\alpha_1,\dots,m_n+\alpha_n)}.\] One has that $M(\alpha)=M\otimes_RR(\alpha)$. Hence, for any (ungraded) $R$-module $M$, we shall denote $M(\alpha):=M\otimes_RR(\alpha)$.
For each $p\in\A^n_{\FF_1}$, $l\in\ZZ$, we shall denote
\[ l_p=(\alpha_1,\dots,\alpha_n), \text{ with } \alpha_i=\left\{\aligned l, &\text{ if } i\in p
\\ 0, &\text{ if } i\notin p\endaligned   \right.\]

For any $p\in\A^n_{\FF_1}$, ${\mathbf x}^p$ denotes the squarefree monomial ${\mathbf x}^p=\prod_{i\in p} x_i$. One has a natural isomorphism of $\ZZ^n$-graded $R$-modules $({\mathbf x}^p)=R(-1_p)$. For any closed subset $K\subseteq \A^n_{\FF_1}$, $I_K$ denotes the Stanley--Reisner ideal associated to $K$, i.e., the ideal generated by the monomials ${\mathbf x}^p$ with $p\notin K$. We shall denote $R_K=R/I_K$, the Stanley--Reisner ring associated to $K$. In particular, for each $p\in\A^n_{\FF_1}$, $R_{C_p}$ is the polynomial ring in the variables $x_i$, $i\in p$.

Let us consider Yanagawa's functor (\cite{Ya01})
\[\pi^\stella\colon \Shv_k(\A^n_{\FF_1})\longrightarrow   R-\text{mod} \] which associates, to each sheaf $F$, a squarefree module $\pi^\stella F$. If $F\to F'$ is a morphism of sheaves, then $\pi^\stella F\to\pi^\stella F'$ is a homogeneous (i.e. $\ZZ^n$-degree preserving) morphism of $\ZZ^n$-graded $R$-modules. We shall not use Yanagawa's original definition. Instead, we shall use an alternative construction given in \cite{ST1}: the functor $\pi^\stella$ is defined as the unique left exact functor   $\pi^\stella\colon  \medskip \Shv_k(\A^n_{\FF_1})\longrightarrow  R-\text{mod}$  that maps $ k_{U_p} $ to $({\mathbf x}^p)$ for any $p\in\A^n_{\FF_1}$ and the natural inclusion $k_{U_q}\hookrightarrow k_{U_p}$ to the natural inclusion $({\mathbf x}^q)\hookrightarrow ({\mathbf x}^p)$ for any $p\leq q$.

Further properties of $\pi^\stella$ are: It is an exact functor and it satisfies:
\[ \pi^\stella (k_K) = R_K\] for any closed subset $K\subseteq\A^n_{\FF_1}$.  Also notice that, $$\aligned \pi^\stella(k_{\{\un\}})=({\mathbf x}^\un)&=\Omega^n_{R/k}\\ {\mathbf x}^\un&\mapsto \di x_1\wedge\dots\wedge\di x_n\endaligned$$ where $\Omega^n_{\A^n_k/k}$ is the module of  highest K\"ahler differentials. 

The functor $\pi^\stella$ has a right adjoint (because it is exact and commutes with direct limits)
\[\pi_\stella\colon   R-\text{mod} \longrightarrow \Shv_k(\A^n_{\FF_1})\] whose stalkwise description is the following: for any $R$-module $M$ and any $p\in \A^n_{\FF_1}$,
\[(\pi_\stella M)_p=\Hom_R(({\mathbf x}^p),M)=M\otimes_R R(1_p)\] and for any $p\leq q$, the morphism $(\pi_\stella M)_p\to (\pi_\stella M)_q$ is obtained by applying $\Hom_R(\quad,M)$ to the inclusion $({\mathbf x}^q)\hookrightarrow ({\mathbf x}^p)$. Then,
$\pi_\stella$ is an exact functor. Moreover, if $M$ is a $\ZZ^n$-graded $R$-module, then $\pi_\stella M$ is a sheaf of $\ZZ^n$-graded $k$-modules (in fact, $R$-modules) on $\A^n_{\FF_1}$.

\begin{prop}[Additivity of $\pi_\stella$]\label{product} Let $R=k[x_1,\dots,x_n]$. For each $i=1,\dots ,n$, let $M_i$ be a $k[x_i]$-module. Then $M_1\otimes_k\cdots\otimes_kM_n$ is an $R$-module and:
\[ \pi_\stella (M_1\otimes_k\cdots\otimes_kM_n)= (\pi_\stella M_1)\boxtimes\cdots\boxtimes (\pi_\stella M_n),\] If each $M_i$ is a $\ZZ$-graded $k[x_i]$-module, then this equality is an isomorphism of $\ZZ^n$-graded sheaves of $k$-modules. 
\end{prop}

\begin{proof} Let $p=(p_1,\dots, p_n)\in\A^n_{\FF_1}$. Then $({\mathbf x}^p)=(x_1^{p_1})\otimes_k\cdots\otimes_k (x^{p_n})$, where $x_i^{p_i}=1$ if $p_i=\0$ and $x_i^{p_i}=x_i$ if $p_i=\un$. Then
\[\aligned & [\pi_\stella (M_1\otimes_k\cdots\otimes_kM_n)]_p = ( M_1\otimes_k\cdots\otimes_kM_n)\otimes_R R(1_p) \\ &= M_1(1_{p_1})\otimes_k\cdots\otimes_k M_n(1_{p_n}), \quad \text{with } 1_{p_i}=\left\{\aligned 0,\text{ if }p_i=\0 
\\ 1, \text{ if } p_i=\un\endaligned\right.
\\ & =  (\pi_\stella M_1)_{p_1} \otimes_k\cdots\otimes_k (\pi_\stella M_n)_{p_n} 
 =[ (\pi_\stella M_1)\boxtimes\cdots\boxtimes (\pi_\stella M_n)]_p.\endaligned\]
\end{proof}

Since $\pi^\stella$ and $\pi_\stella$ are exact functors, they induce functors in the derived categories
\[ \pi^\stella\colon D(\A^n_{\FF_1})\to D(R)\quad,\quad \pi_\stella\colon D(R)\to D(\A^n_{\FF_1})\] which are still adjoint.
}\end{cosa}

\section{Main theorem and applications}

For each $ M\in D(R)$, let us denote
\[ D( M)=\RR\Hom_R^\pun( M, \Omega^n_{R/k}).\]  Recall that for any $F\in D(\A^n_{\FF_1})$, $D(F)$ denotes the complex $\RR\HHom_{\A^n_{\FF_1}}^\pun(F,k_{\{\un\}})$. Then:

\begin{thm}\label{Hochster} For any $ M\in D (R)$, and any pointwise perfect $F\in D_{\text{\rm perf}}(\A^n_{\FF_1})$, one has
\[ \RR\Hom_R^\pun( M, \pi^\stella F)=\RR\Hom_{\A^n_{\FF_1}}^\pun(D(F),\pi_\stella D(M)).\]
\end{thm}

\begin{proof} Let us first see that $\pi^\stella F $ is a perfect complex of $R$-modules. Let us consider the standard resolution (see \cite{Sanchoetal})
\[ 0\to C_nF\to \cdots\to C_1 F \to C_0F\to F\to 0\] with $C_iF=\underset{p_0<\cdots <p_i}\bigoplus (F_{p_0})\otimes_k k_{U_{p_i}}$. If suffices to see that $\pi^\stella C_iF$ is perfect. Hence, we are reduced to prove that $\pi^\stella(E\otimes_k k_{U_p})$ is perfect, for any $E\in D(k)$ perfect and $p\in \A^n_{\FF_1}$. In this case,  $\pi^\stella (E\otimes_k k_{U_p})= E \otimes_k \pi^\stella ( k_{U_p}) = E \otimes_k ({\mathbf x}^p)$ is perfect.  

Now, since $\pi^\stella F $ is perfect and $\Omega^n_{R/k}$ is an invertible $R$-module, one has $$\RR\Hom_R^\pun( M,\pi^\stella F)= \RR\Hom_R^\pun(D(\pi^\stella F),D( M)).$$ 
Moreover, $D(\pi^\stella F)=\pi^\stella D(F)$ (see \cite[Thm 5.1]{ST1}). Then
\[   \RR\Hom_R^\pun( M,\pi^\stella F)= \RR\Hom_R^\pun ( \pi^\stella D( F),D( M)) 
  =  \RR\Hom_{\A^n_{\FF_1}}^\pun ( D( F), \pi_\stella D( M)).  \]
\end{proof}




\begin{rem} {\rm As we shall see, Hochster's formulae appear when one is able to decompose $\pi_\stella D(M)$, and this highly depends on of the choice of $M$. A  first example will be Theorem \ref{Hochster-injective}.}
\end{rem}

\begin{lem}\label{decomposition} Let $J$ be a sheaf of $k$-vector spaces on a finite poset $X$. If $J$ is injective, then it decomposes   as
\[ J\simeq \underset{x\in X}\bigoplus\, J(x)\otimes_kk_{C_x},\] with $J(x)=\Hom_X(k_{\{x\}},J)$.
\end{lem}

\begin{proof} We proceed by induction on $r=$ number of points of $X$. If  $r=1$, it is immediate. Assume  $r>1$, let $x_0\in X$ be a closed point and $U:=X\negmedspace-\negmedspace\{x_0\}\overset j\hookrightarrow X$. Taking $\HHom_X(\quad,J)$ in the exact sequence $0\to k_U\to k\to k_{\{x_0\}}\to 0$ (where $\HHom_X$ is the sheaf of homomorphisms), we obtain
\[ 0\to J(x_0)\otimes_k k_{\{x_0\}}\to J\to j_*J_{\vert U}\to 0.\] Since $J(x_0)\otimes_k k_{\{x_0\}}$ is injective, this sequence splits, so $J\simeq J(x_0)\otimes_k k_{\{x_0\}}\oplus j_*J_{\vert U}$. One concludes by induction because $J_{\vert U}$ is an injective sheaf on $U$.
\end{proof}
\begin{thm}\label{Hochster-injective} Assume that $k$ is a field.  Let $ M\in D (R)$ such that $D( M)= I[r]$ for some injective $R$-module $ I$ and some integer $r$. Then, for any pointwise perfect $F\in D_{\text{\rm perf}}(\A^n_{\FF_1})$ one has a  (non-canonical)  decomposition
\[ \Ext_R^i( M,  \pi^\stella F)\simeq \underset{p\in\overline{\supp(F)}}\bigoplus\, H^{i+n+r-c_p}_p(U_p,F)\otimes_k \Hom_R(R_{ C_p },  I)(1_p),\qquad c_p=\dim C_p.\]  
If $M$ is $\ZZ^n$-graded, this decomposition is an isomorphism of $\ZZ^n$-graded $k$-vector spaces.
\end{thm}

\begin{proof}

Since $\pi^\stella$ is exact, $J:=\pi_\stella I$ is an injective sheaf on $\A^n_{\FF_1}$, so it decomposes (by Lemma \ref{decomposition})
\[ J\simeq  \underset{p\in\A^n_{\FF_1}}\bigoplus\, J(p)\otimes_k k_{C_p}\] with $J(p)=H^0_p(U_p,J)=\Hom_{\A^n_{\FF_1}}(k_{\{p\}}, J)$. Then
\[\aligned  &\RR\Hom_R^\pun(\M,\pi^\stella F)\overset{\ref{Hochster}}=\RR\Hom_{\A^n_{\FF_1}}(D(F),\pi_\stella D( M))  = \RR\Hom_{\A^n_{\FF_1}}(D(F),  \pi_\stella  I)[r] 
\\ & \simeq \underset{p\in\A^n_{\FF_1}}\bigoplus\,  \RR\Hom_{\A^n_{\FF_1}}^\pun(D(F), J(p)\otimes_k k_{C_p})[r].\endaligned\]
Now, since $k$ is a field, $\overset\LL\otimes_k=\otimes_k$, and then $ \RR\Hom_{\A^n_{\FF_1}}^\pun(D(F), J(p)\otimes_k k_{C_p})\overset{\ref{formulaD(F)}}=  \RR\Gamma_p(U_p,F)[n-c_p]\otimes_k J(p) $, which is zero if $p\notin \overline{\supp(F)}$. To conclude, it suffices to see that $J(p)= \Hom_R(R_{C_p}, I)(1_p)$. By definition
\[ J(p)=\Hom_{\A^n_{\FF_1}}(k_{\{p\}},J)=\Hom_R(\pi^\stella k_{\{p\}},  I).\]
Applying $\pi^\stella$ to the exact sequence 
\[ \underset{i\notin p} \bigoplus \, k_{U_{p\cup {\{i\}}}} \to k_{U_p}\to k_{\{p\}}\to 0 \]  one obtains that $\pi^\stella k_{\{p\}}= R_{ C_p }(-1_p)$, and then $J(p)=\Hom_R(R_{ C_p },  I)(1_p)$.
\end{proof}

\begin{prop}\label{residuo-descomp} Let $R=k[x_1,\dots,x_n]$, $\mm=(x_1,\dots ,x_n)$ and  $\mm_l =(x_1^l,\dots ,x_n^l)$. One has
\begin{enumerate} \item[(a)] 
$ \pi_\stella H^n_\mm(\Omega^n_{R/k}) = \underset{p\in\A^n_{\FF_1}}\bigoplus\, H^{c_p}_\mm(R_{C_p}) \otimes_k  k_{C_p}, \text{ with } c_p=\dim C_p.$
\item[(b)]  $\pi_\stella \Hom_R(R/\mm_l, H^n_\mm(\Omega^n_{R/k})) =  \underset{q\geq p\in\A^n_{\FF_1}}\bigoplus\, (R_{C_p}/\mm_{l-1}) \otimes_k  k_{C_q}\otimes_k k_{U_{q-p}}  ( l_q).$ 
\end{enumerate}
\end{prop}

\begin{proof} For $n=1$ it is a direct computation. The general case follows from the case $n=1$ and the additivity of $\pi_\stella$ (Proposition \ref{product}). Let us give the details.

 (a). Let us first consider the case $n=1$. In this case,   $$R=k[x] , \quad \mm=(x) ,\quad  I:=H^1_\mm(\Omega_{k[x]/k})= k[x^{-1}]\frac{\di x}x.$$ 
One has a canonical isomorphism $$  k  =\Hom_R(R/\mm,I),\quad  1 \mapsto \frac{\di x}x $$ and the natural inclusion $k\hookrightarrow I$ has a retract, the residue map:
\[ \aligned \Res_0\colon I&\longrightarrow k \\ \omega &\mapsto \Res_0(\omega).\endaligned\]

Now, $\pi_\stella I$ is the sheaf on $\A^1_{\FF_1}$ given by
\[ (\pi_\stella I)_\0 = I,\quad (\pi_\stella I)_\un= \Hom_R((x),I)=I(1)\] and the morphism $I=(\pi_\stella I)_\0
\to (\pi_\stella I)_\un=I(1)$ is the multiplication by $x$. Taking $\HHom_{\A^n_{\FF_1}}(\quad,\pi_\stella I)$ in the exact sequence of sheaves on $\A^1_{\FF_1}$ \[ 0\to k_{\{\un\}}\to k\to k_{\{\0\}}\to 0\] we obtain
the exact sequence
\begin{equation}\label{pre-split} 0\to k_{\{\0\}}  \to \pi_\stella I\to I(1) \to 0\end{equation} where $I(1) $ is the constant sheaf $I(1)$ on $\A^1_{\FF_1}$; indeed, let us first see that $\HHom_{\A^n_{\FF_1}}(k_{\{\0\}},\pi_\stella I)=k_{\{\0\}}$. Since $\HHom_{\A^n_{\FF_1}}(k_{\{\0\}},\pi_\stella I)$ is supported at the closed point $\0$, one has $\HHom_{\A^n_{\FF_1}}(k_{\{\0\}},\pi_\stella I)=M\otimes_k k_{\{\0\}}$ for some $k$-module $M$. Taking the stalk at $\0$ (i.e., global sections), one obtains $M=\Hom_{\A^n_{\FF_1}}(k_{\{\0\}},\pi_\stella I)=\Hom_R(\pi^\stella k_{\{\0\}}, I) = \Hom_R(R/\mm, I) =k$. Now, let us see that $\HHom_{\A^n_{\FF_1}}(k_{\{\un\}},\pi_\stella I) = I(1)$. Let us denote $j\colon \{\un\}\hookrightarrow \A^1_{\FF_1}$. Then
\[  \HHom_{\A^n_{\FF_1}}(k_{\{\un\}},\pi_\stella I)  = j_*j^{-1}\pi_\stella I = j_*(\pi_\stella I)_\un = j_*I(1)= I(1).\]
The sequence \eqref{pre-split} splits canonically, since a retract of $k_{\{\0\}}  \to \pi_\stella I$ is given by the morphism $\pi_\stella I \to  k_{\{\0\}} $ corresponding to the residue map $\Res_0\colon I\to k$ under the equality
\[ \aligned \Hom_{\A^1_{\FF_1}}(\pi_\stella I, k_{\{\0\}}) &= \Hom_k(I,k) \\ f&\mapsto f_\0=\Gamma(\A^1_{\FF_1},f)\endaligned.\] Also, the section of $\pi_\stella I\to I(1) $ is the morphism $I(1)  \to  \pi_\stella I$ corresponding to the morphism $I(1)\overset{\cdot 1/x}\to I$ under the equality \[ \aligned \Hom_{\A^1_{\FF_1}}(I(1)  ,\pi_\stella I) &= \Hom_k(I(1),I) \\ f&\mapsto f_\0=\Gamma(\A^1_{\FF_1},f)\endaligned.\]
Since the sequence splits, one has: $\pi_\stella I = k_{\{\0\}}\oplus I(1)$ and (a) is proved for  $n=1$ (notice that $I(1)=H^1_\mm(R)$, since $\Omega_{R/k}(1)=R$).

The general case of (a) follows from the case $n=1$ by additivity, as follows: now, $R=k[x_1,\dots,x_n]$ and one has an isomorphism 
\[ H^n_\mm(\Omega^n_{R/k})= H^1_{(x_1)}(\Omega_{k[x_1]/k})\otimes_k \cdots \otimes_k H^1_{(x_n)}(\Omega_{k[x_n]/k}))\] and then (Proposition \ref{product})
\[ \pi_\stella H^n_\mm(\Omega^n_{R/k}) = (k_{\{\0\}}\oplus H^1_{(x_1)}(k[x_1]))\boxtimes\cdots \boxtimes (k_{\{\0\}}\oplus H^1_{(x_n)}(k[x_n])).\]  This yields the result, taking into account that
\[ k_{C_p}=\underset{i\notin p}\bigotimes\, k_{H_i} = \underset{i\notin p}\bigotimes\, \phi_i^{-1}k_{\{\0\}},\qquad \phi_i\colon\A^n_{\FF_1}\to\A^1_{\FF_1} \text{ the } i\text{-th projection,}\] and
\[ H^{c_p}_\mm(R_{C_p})= \underset{i\in p}\bigotimes\, H^1_{(x_i)}(k[x_i]).\]

(b). Again, let us first consider the case $n=1$. In this case, $R=k[x]$, $\mm_l=(x^l)$ and we denote $$I_l:=\Hom_R(R/\mm_l, H^1_\mm(\Omega_{R/k}))=(R/\mm_l)^*=k\frac{\di x}x\oplus\cdots\oplus k\frac{\di x}{x^l}.$$  The sheaf $\pi_\stella I_l$ on $\A^1_{\FF_1}$ is described by
\[ (\pi_\stella I_l)_\0=I_l,\quad (\pi_\stella I_l)_\un=I_l(1),\quad I_l=(\pi_\stella I)_\0\overset{\cdot x}\to  (\pi_\stella I)_\un=I_l(1).\] Let us see that
\[ \pi_\stella I_l= k_{\{\0\}}\oplus k_{\{\un\}}(l)\oplus (R/\mm_{l-1})(l).\]
 Taking $\HHom_{\A^n_{\FF_1}}(\quad,\pi_\stella I_l)$ in the epimorphism $k\to k_{\{\0\}}$ one obtains (as in the proof of \ref{residuo-descomp}) $0\to k_{\{\0\}}\to \pi_\stella I_l$, which has a retract (again, given by the residue map), and then $\pi_\stella I_l= k_{\{\0\}}\oplus C$, with $C=\pi_\stella I_l/k_{\{\0\}}$. The sheaf $C$ is described as
 \[ C_\0=I_l/I_1=I_{l-1}(1), \quad C_\un=I_l(1)\] and the morphism $C_\0\overset{\cdot x}\to C_\un$ is injective, with cokernel $I_l(1)/x\cdot I_{l-1}=I_1(l)=k(l)$. Hence, we have an exact sequence
 \[ 0\to I_{l-1}(1)\to C \to k_{\{\un\}} (l)\to 0\] which splits, because one has a section $k_{\{\un\}} (l) \to C$, induced by the morphism $k(l)\to C_\un$, $1\mapsto (\di x)/x^l$. This ends the proof of (b) for $n=1$, since $I_{l-1}(1)=(R/\mm_{l-1})(l)$. As before, the general case is obtained by additivity, as follows.
 
 Now, $R=k[x_1,\dots ,x_n]$. Then
 \[  (R/\mm_l)^*=  (k[x_1]/(x_1^l))^*\otimes_k\cdots\otimes_k  (k[x_n]/(x_n^l))^*\] hence,
 \[\aligned  \pi_\stella (R/\mm_l)^* &= \bigotimes_{i=1}^n\, \phi_i^{-1} [k_{\{\0\}}\oplus k_{\{\un\}}(l )\oplus (k[x_i]/(x_i^{l-1})  (l)],\quad \phi_i\colon \A^n_{\FF_1}\to \A^1_{\FF_1} \text { the } i\text{-th projection},
 \\ & = \underset{p\in\A^n_{\FF_1}}\bigoplus \, \left[ \underset{i\in \widehat p}\bigotimes\, (k_{H_i}\oplus k_{U_i}(l ))\, 
 \underset{i\in   p}\bigotimes\, (k[x_i]/(x_i^{l-1}))(l)\right]. 
 \endaligned\]
 Now, $$\underset{i\in   p}\bigotimes\, (k[x_i]/(x_i^{l-1}))(l)= (R_{C_p}/\mm_{l-1})(l_p)$$ and
 \[ \aligned \underset{i\in \widehat p}\bigotimes\, (k_{H_i}\oplus k_{U_i}(l ))& =\underset{q\leq \widehat p}\bigoplus \,  \left[ \underset{i\in q}\bigotimes\, k_{H_i} \otimes_k  \underset{i\in \widehat p-q}\bigotimes\, k_{U_i}(l )\right] =  \underset{q\leq \widehat p} \bigoplus \, k_{C_{\widehat q}}\otimes_k k_{U_{\widehat p -q}}(l_{\widehat p-q})
 \\ & =  \underset{q\geq  p}\bigoplus \, k_{C_{q}}\otimes_k k_{U_{q-p}}(l_{ q-p})
\endaligned. \] This ends the proof of (b).
 \end{proof}

If $k$ is a field and $F$ is a sheaf of finite dimensional $k$-vector spaces on $\A^n_{\FF_1}$, Yanagawa computes  (in \cite[Thm. 3.3, Thm. 3.5]{Ya03} ) the degree-wise components of $H^i_\mm(\pi^\stella F)$   as   cohomology groups with compact support of a certain sheaf $(\pi^\stella F)^+$ on the geometric realization of $\A^n_{\FF_1}$. The next theorem generalizes Yanagawas's results for any ring $k$ and any pointwise perfect complex $F$ on $\A^n_{\FF_1}$, but   computing $H^i_\mm(\pi^\stella F)$ in terms of local cohomology groups of $F$ (that is, we do not go through the geometric realization). See also \cite{BBR}.

 \begin{thm} \label{cor2} For any pointwise perfect $F\in D_{\text{\rm perf}}(\A^n_{\FF_1})$, one has a canonical decomposition of $\ZZ^n$-graded $k$-modules:
\[ H^i_\mm (\pi^\stella F) = \underset{p\in \overline{\supp(F)}}\bigoplus\, H^{i-c_p}_p(U_p,F)\otimes_k H^{c_p}_\mm(R_{C_p}), \quad c_p=\dim C_p.  \] Taking $F=k_K$, with $K$ a closed subset of $\A^n_{\FF_1}$, one obtains Hochster's formula
\[ H^i_\mm (R_K) = \underset{p\in K}\bigoplus\,  H^{i-c_p-1}_{\text{\rm red}}(\Link_K(p),k)\otimes_k  H^{c_p}_\mm(R_{C_p}). \]

\end{thm}

\begin{proof} By Theorem \ref{Hochster} one has
\[ \RR\Hom_R^\pun(R/\mm^l,\pi^\stella F)= \RR\Hom_{\A^n_{\FF_1}}^\pun(D(F),\pi_\stella D(R/\mm^l)).\] Taking direct limit on $l$, one obtains 
\begin{equation}\label{Hoch} \RR\Gamma_\mm(\pi^\stella F) = \RR\Hom_{\A^n_{\FF_1}}^\pun(D(F),\pi_\stella H^n_\mm(\Omega^n_{R/k}))[-n].\end{equation} Indeed,
\[\aligned &\RR\Gamma_\mm(\pi^\stella F)  = \ilim{l} \RR\Hom_R^\pun(R/\mm^l,\pi^\stella F) = \ilim{l}\RR\Hom_{\A^n_{\FF_1}}^\pun(D(F),\pi_\stella D(R/\mm^l)) 
\\ & = \ilim{l}\RR\Hom_{\A^n_{\FF_1}}^\pun(\pi^\stella D(F), D(R/\mm^l)) \overset{\pi^\stella D(F)\text{ is perfect}}{====}  \RR\Hom_{\A^n_{\FF_1}}^\pun(\pi^\stella D(F), \ilim{l} D(R/\mm^l))
\\ & =    \RR\Hom_{R}^\pun(\pi^\stella D(F), \RR\Gamma_\mm \Omega^n_{R/k})= \RR\Hom_{\A^n_{\FF_1}}^\pun( D(F), \pi_\stella\RR\Gamma_\mm \Omega^n_{R/k}) \endaligned\] and we conclude because $\RR\Gamma_\mm \Omega^n_{R/k}=H^n_\mm(\Omega^n_{R/k})[-n]$.

Now, combining  \eqref{Hoch} and Proposition \ref{residuo-descomp},(a), we obtain 
\[  \RR\Gamma_\mm(\pi^\stella F)= \underset{p\in\A^n_{\FF_1}}\bigoplus\, \RR\Hom_{\A^n_{\FF_1}}^\pun(D(F), k_{C_p}\otimes_k H^{c_p}_\mm(R_{C_p}))[-n]\] and one concludes because
\[ \RR \Hom_{\A^n_{\FF_1}}^\pun( D(F), k_{C_p}\otimes_k H^{c_p}_\mm(R_{C_p})) \overset{\eqref{formulaD(F)}}  = \RR\Gamma_p(U_p,F)[n-c_p] \otimes_k H^{c_p}_\mm(R_{C_p}), \] which is zero if $p\notin \overline{\supp(F)}$.
Finally, for $F=k_K$, the formula follows from the equality $H^i_p(U_p, k_K)=H^{i-1}_\text{red}(\Link_K(p),k)$. Indeed, let us denote $U_p^*=U_p-\{p\}$ and $\RR\Gamma(U_p,\quad)$ (resp. $\RR\Gamma(U_p^*,\quad)$) the right derived functor of the functor of sections in $U_p$ (resp. in $U_p^*$); one has the exact triangle 
\[\displaystyle \RR\Gamma_p(U_p,k_K)\to \underset{\underset{\text{\small $k$}} \Vert  } {\RR\Gamma(U_p,k_K)}\to \underset{\underset{ \text{\small $\RR\Gamma (U_p^*\cap K,k)$}}\Vert} {\RR\Gamma(U_p^*,k_K)}\] and then $\RR\Gamma_p(U_p,k_K)= \RR\Gamma_{\text{red}} (U_p^*\cap K,k)[-1]$;  one concludes by the homeomorphism $\Link_K(p)\simeq U_p^*\cap K$.
\end{proof}


\begin{rems}\label{residuo}{\em (1) If $k$ is a field, one can immediately compute the Hilbert series of $H^i_\mm(\pi^\stella F)$ from Theorem \ref{cor2}, since $H^{c_p}_\mm(R_{C_p})=\underset{i\in p}\bigotimes\, \frac{1}{x_i}k[x_i^{-1} ]$. In particular,  one obtains Hochster's formula for the Hilbert series of the Stanley--Reisner ring $R_K$ (\cite[Section II.4]{St}).

(2) The $R$-module structure of the right hand term $$\underset{p\in \overline{\supp(F)}}\bigoplus\, H^{i-c_p}_p(U_p,F)\otimes_k H^{c_p}_\mm(R_{C_p})$$ can be tracked from the proof, as follows. First, let us consider the case $n=1$, and the splitting $\pi_\stella I=k_{\{\0\}}\oplus H^1_{(x)}(k[x])$. The multiplication $I\overset{\cdot x}\to I$ induces a morphism of sheaves $\pi_\stella I\overset{\pi_\stella(\cdot x)}\longrightarrow \pi_\stella I$, which, via the splitting, becomes the morphism
\[ k_{\{\0\}}\oplus H^1_{(x)}(k[x]) \to k_{\{\0\}}\oplus H^1_{(x)}(k[x])\] defined by the matrix 
\[\begin{pmatrix} 0 & \Res_0\\ 0 & x\end{pmatrix}\] where $x\colon H^1_{(x)}(k[x])\to H^1_{(x)}(k[x])$ is just the multiplication by $x$ and $\Res_0\colon H^1_{(x)}(k[x])\to k_{\{\0\}}$ is the morphism corresponding to $\Res_0\colon H^1_{(x)}(k[x])\to k$, $P\mapsto \Res_0(P\di x)$, under the equaltity $\Hom_{\A^n_{\FF_1}}(H^1_{(x)}(k[x]), k_{\{\0\}})=\Hom_k(H^1_{(x)}(k[x]),k)$.

Then, the multiplication by $x_j$ in the direct sum $\underset{p\in \overline{\supp(F)}}\bigoplus\, H^{i-c_p}_p(U_p,F)\otimes_k H^{c_p}_\mm(R_{C_p})$ is, on each factor $H^{i-c_p}_p(U_p,F)\otimes_k H^{c_p}_\mm(R_{C_p})$, the morphism:

(a) If $j\in p$, it is the morphism (let us denote $p'=p-\{j\}$)
\[ \aligned  H^{i-c_p}_p(U_p,F)\otimes_k H^{c_p}_\mm(R_{C_p}) \longrightarrow & H^{i-c_p}_p(U_p,F)\otimes_k H^{c_p}_\mm(R_{C_p}) \\ &\oplus H^{i-c_p+1}_{p'}(U_{p'},F)\otimes_k H^{c_p-1}_\mm(R_{C_{p'}}) \endaligned\] whose first component is just multiplication by $x_j$ and the second one is the morphism 
\[ H^{i-c_p}_p(U_p,F)\otimes_k H^{c_p}_\mm(R_{C_p})\overset{\delta_{pp'}^i\otimes \Res_{x_j=0}} \longrightarrow H^{i-c_p+1}_{p'}(U_{p'},F)\otimes_k H^{c_p-1}_\mm(R_{C_{p'}}) \] where \[ \delta_{pp'}^i\colon H^{i-c_p}_p(U_p,F)\to H^{i-c_p+1}_{p'}(U_{p'},F)\] is the natural connecting and
\[ \Res_{x_j=0}\colon H^{c_p}_\mm(R_{C_p}) \to H^{c_p-1}_\mm(R_{C_{p'}})\] is the morphism induced by $\Res_0\colon H^1_\mm(k[x_j])\to k$, taking into account the equality
\[ H^{c_p}_\mm(R_{C_p})= H^{c_p-1}_\mm(R_{C_{p'}})\otimes_k H^1_\mm(k[x_j]) \] induced by the equality $R_{C_p}=R_{C_{p'}}\otimes_k k[x_j].$

(b) The null morphism if $j\notin p$.
}\end{rems}

In \cite[Thm. 1]{Mi}, Miyazaki computes the degree-wise components of $\Ext^i_R(R/\mm_l,R_K)$ for any closed subset $K$ of $\A^n_{\FF_1}$. The next theorem generalizes Miyazaki's result for any pointwise perfect complex $F$ on $\A^n_{\FF_1}$. It is a consequence of Theorem \ref{Hochster} and Proposition \ref{residuo-descomp}, (b).


\begin{thm}\label{Hochster-extens} Let us denote $\mm_l =(x_1^l,\dots ,x_n^l)$. For any pointwise perfect $F\in D_{\text{\rm perf}}(\A^n_{\FF_1})$, one has a canonical decomposition of $\ZZ^n$-graded $k$-modules:
\[ \Ext^i_R(R/\mm_l,\pi^\stella F)= \underset{p\leq q\in\overline{\supp(F)}}\bigoplus \,     H^{i-c_q}_q(U_q,F)\otimes_k  (R_{C_p}/\mm_{l-1})(l_q ).
\] Taking $F=k_K$, with $K$ a closed subset of $\A^n_{\FF_1}$,one obtains 
\[ \Ext^i_R(R/\mm_l,R_K)= \underset{p\leq q\in K}\bigoplus \,     H^{i-c_q-1}_\text{\rm red}(\Link_K(q),k)\otimes_k  (R_{C_p}/\mm_{l-1})(l_q ), 
\] and taking the $\alpha$-degree component in this equality (with $\alpha\in\ZZ^n$) one obtains Miyazaki's result (\cite[Thm. 1]{Mi}).
\end{thm}

\begin{proof} By Theorem \ref{Hochster}
\[ \RR\Hom_R^\pun(R/\mm_l,\pi^\stella F)=\RR\Hom_{\A^n_{\FF_1}}^\pun(D(F),\pi_\stella D(R/\mm_l)).\] Now, $D(R/\mm_l)=\Hom_R(R/\mm_l, H^n_\mm(\Omega^n_{R/k})) [-n]$, so  $$\pi_\stella D(R/\mm_l)= \pi_\stella \Hom_R(R/\mm_l, H^n_\mm(\Omega^n_{R/k}))[-n].$$   By Proposition \ref{residuo-descomp}, (b), 
\[\aligned  \RR\Hom_R^\pun(R/\mm_l,\pi^\stella F)&=  \underset{q\geq p}\bigoplus \, \RR\Hom_{\A^n_{\FF_1}}^\pun(D(F),k_{C_q}\otimes_k k_{U_{q-p}}\otimes_k (R_{C_p}/\mm_{l-1})(l_q ))[-n]
\\ & = \underset{q\geq p}\bigoplus \, \RR\Hom_{\A^n_{\FF_1}}^\pun(D(F),k_{C_q}\otimes_k k_{U_{q-p}})\otimes_k (R_{C_p}/\mm_{l-1})(l_q )[-n]
\endaligned\] where the last equality follows because $R_{C_p}/\mm_{l-1}$ is a finite and free $k$-module. Then we are reduced to prove  the following formula (for any $x\geq y\in\A^n_{\FF_1}$):
\[ \RR\Hom_{\A^n_{\FF_1}}^\pun(D(F),k_{C_x}\otimes_k k_{U_y}) = 
 \RR\Gamma_x( U_x,F)  [n-c_x] \quad(\text{which is zero if } x\notin\overline{\supp(F)}).\]
Let us denote $j\colon U_y\hookrightarrow\A^n_{\FF_1}$, $C_x^{y}$ the closure of $x$ in $U_y$ (in other words, $C_x^{y}=C_x\cap U_y=j^{-1}(C_x)$). Notice that $U_y=\A^{n-c_y}_{\FF_1}, q\mapsto q-y $, and $j^{-1}k_{\{\un\}}=k_{\{\un\}}$, so $j^{-1}D(F)=D(j^{-1}F)$. Then
 $$\aligned   \RR\Hom_{\A^n_{\FF_1}}^\pun(D(F),k_{C_x}\otimes_k k_{U_y} )&= \RR\Hom_{\A^n_{\FF_1}}^\pun(D(F),j_!j^{-1}k_{C_x}) \\ &=   \RR\Hom_{\A^{n-c_y}_{\FF_1}}^\pun(j^{-1}D( F), k_{C_x^{y}}) \quad (\text{adjunction  between } j^{-1} \text{ and } j_!)\\ &= \RR\Hom_{\A^{n-c_y}_{\FF_1}}^\pun(D(j^{-1}F), k_{C_x^{y}})
 \\ & \overset{\eqref{formulaD(F)}} = \RR\Gamma_x(U_x,F)[n-c_y-\dim C_x^{y}]\endaligned$$ and we conclude because $\dim C_x^y = c_x-c_y$.  
 
\end{proof}

\subsection{Comparison with other Hochster's type formulae}

Let us compare our results with other Hochster's type formulae obtained more recently with a different approach (see \cite{BBR} and specially \cite{ABZ}). Let us first consider Hochster's formula for local cohomology. Let $K$ be a finite simplicial complex and $R_K$ its associated Stanley-Reisner ring. The classical Hochster's formula computes $H^i_\mm(R_K)$, when $k$ is a field, and Yanagawa generalizes this result for any squarefree module (over a field) and our Theorem \ref{cor2} generalizes Yanagawa's result for any ring $k$ and any complex of squarefree modules. A different approach to compute  $H^i_\mm(R_K)$, which goes back to Yuzvinsky (\cite{Yuz}), is to see $R_K$ as the global sections of a sheaf of $k$-algebras on the poset $K^{\rm op}$: for each $p\in K$, let $R_p=R_{C_p}$ and, for any $p\leq q$, $R_q\to R_p$  the natural epimorphism. We have then a sheaf $\CR_K$ on $K^{\rm op}$, defined by $(\CR_K)_p=R_p$. The global sections of $\CR_K$ is the inverse limit of the ${R_p}'s$, which amounts to be $R_K$: $$\Gamma(K^{\rm op},\CR_K)=\plim{p\in K} R_p=R_K.$$

From this point of view, Hochster's formula gives a description of the local cohomology of the limit $\plim{p\in K} R_p$ in terms of the local cohomology of the ${R_p}'s$. This  has been generalized by replacing the simplicial complex $K$ by an arbitrary poset $P$ and considering a sheaf of $k$-algebras $\CR$ on $P$ (as in \cite{BBR}) or, more generally, a sheaf $\CR$ of $A$-modules on $P$ (as in \cite{ABZ}). Under certain ``degeneracy'' conditions, one obtains a Hochster's type result: the local cohomology of the limit $\plim{} R_p$ is computed in terms of the local cohomology of the ${R_p}'s$. But Yanagawa's result (or our Theorem \ref{cor2}) is not obtained from this perpective. 

Regarding our Theorem \ref{Hochster-extens}, or Miyazaki's result, something similar happens. There are Hochster's type results for $\Ext$'s, obtained in \cite{ABZ}, again under the perspective of comparing the $\Ext$ of a limit $\plim{p\in P}R_p$ with the $\Ext$ of the ${R_p}'s$. But these results do not yield Miyazaki's result or our Theorem \ref{Hochster-extens}. For example, let $\mm$ be a maximal ideal of a noetherian ring $A$, $J$ an $\mm$-primary ideal (for example, $J=\mm^l$) and $P_p$  an inverse system  of finitely generated $A$-modules; then $\Ext^i_A(A/J, P_p)$ is a finitely generated module over the Artinian ring $A/J$, and then the degeneracy condition of \cite[Thm. 7.1]{ABZ}

$$\Hom_A(\Ext^{d_p}_A(A/J, P_p), \Ext^{d_q}_A(A/J, P_q)=0$$ is never satisfied if both $\Ext^{d_p}_A(A/J, P_p)$ and $\Ext^{d_q}_A(A/J, P_q)$ are non zero.

\bigskip

\noindent Statements and Declarations\medskip

\noindent{\bf Declarations of interest:} none.

\noindent{\bf Data availability statement.} Data sharing is not applicable to this article as no datasets were generated or analyzed during the current study.

\end{document}